\documentclass[12pt,twoside]{amsart}
\usepackage[all]{xy}
\usepackage{graphicx}
\usepackage{color}

\title{Toric Fano contractions associated to long extremal rays} 
\author[Osamu Fujino]{Osamu Fujino${}^*$}
\thanks{*Partially 
supported by JSPS KAKENHI Grant 
Numbers JP16H03925, JP16H06337. }
\author[Hiroshi Sato]{Hiroshi Sato${}^{**}$}
\thanks{**Partially supported by JSPS KAKENHI 
Grant Number JP18K03262.}

\dedicatory{Dedicated to Professor Shoetsu Ogata on 
the occasion of his sixtieth birthday}

\date{2018/8/13, version 0.26}
\address{Department of Mathematics, Graduate School of 
Science, Osaka University, Toyonaka, Osaka 560-0043, Japan}
\email{fujino@math.sci.osaka-u.ac.jp}
\address{Department of Applied Mathematics, Faculty of Sciences, 
Fukuoka University, 8-19-1, Nanakuma, Jonan-ku, Fukuoka 814-0180, Japan}
\email{hirosato@fukuoka-u.ac.jp}

\makeatletter
    
    \@addtoreset{equation}{section}
\makeatother

\newcommand{\Pic}[0]{{\operatorname{Pic}}}
\newcommand{\mult}[0]{{\operatorname{mult}}}
\newcommand{\NE}[0]{{\operatorname{NE}}}
\newcommand{\G}[0]{{\operatorname{G}}}

\newtheorem{thm}{Theorem}[section]

\newtheorem{cor}[thm]{Corollary}
\newtheorem{prop}[thm]{Proposition}

\theoremstyle{definition}
\newtheorem{ex}[thm]{Example}

\newtheorem{rem}[thm]{Remark}
\newtheorem*{ack}{Acknowledgments}       

\begin{document}
\bibliographystyle{amsalpha+}

\begin{abstract}
We show that 
a toric Fano contraction associated to 
an extremal ray whose length is greater than 
the dimension of its fiber 
is a projective space bundle. 
\end{abstract}

\maketitle

\tableofcontents
\section{Introduction} 

Let $X$ be a smooth projective variety defined over an algebraically closed 
field $k$ of arbitrary characteristic. 
In his epoch-making paper (see \cite{mori}), Shigefumi Mori 
established the following famous cone theorem 
\[
\overline\NE(X)=\overline\NE(X)_{K_X\geq 0}+\sum R_j, 
\]
where $\overline\NE(X)$ denotes the Kleiman--Mori cone of $X$ and 
each $R_j$ is called a $K_X$-negative extremal ray of $\overline\NE(X)$. 
By the original proof of the above cone theorem, which 
is based on Mori's bend and break technique 
to create rational curves, we know that for each $K_X$-negative extremal 
ray $R$ there exists a (possibly singular) rational curve $C$ on 
$X$  such that the numerical equivalence class of $C$ spans $R$ and 
\[
0<-K_X\cdot C\leq \dim X+1
\]
holds. 

Let $X$ be a $\mathbb Q$-Gorenstein projective algebraic 
variety for which 
the cone theorem holds. 
Then for a $K_X$-negative extremal ray $R$ of $\overline\NE(X)$, we 
put 
\[
l(R):=\min_{[C]\in R}(-K_X\cdot C)
\]
and call it the {\em{length}} of $R$. 
We have already known that $l(R)$ is 
an important invariant and that some conditions on $l(R)$ 
determine the structure of the associated extremal contraction. 

In this paper, we are interested in the case where 
$X$ is a toric variety. 
We note that $\NE(X)=\overline\NE(X)$ holds 
when $X$ is a projective toric variety. 
This is because $\NE(X)$ is a rational polyhedral cone. 
We also note that the cone theorem holds for 
$\mathbb Q$-Gorenstein projective toric varieties 
without any extra assumptions. 

From now on, we will only treat $\mathbb{Q}$-factorial 
projective toric varieties defined over 
an algebraically closed field $k$ of arbitrary characteristic for simplicity. 

\medskip 

For a $\mathbb{Q}$-factorial projective toric $n$-fold $X$ of 
Picard number $\rho(X)=1$, there exists the unique 
extremal ray of $\NE(X)$. In this case, the following statement holds. 

\begin{thm}[{\cite[Proposition 2.9]{fujino-notes} and 
\cite[Proposition 2.1]{fujino-osaka}}]\label{rho1}
Let $X$ be a $\mathbb{Q}$-factorial projective toric $n$-fold 
of Picard number $\rho(X)=1$ with 
$R=\NE(X)$. Then, the following statements hold.
\begin{enumerate}
\item 
If $l(R)>n$, then $X\simeq \mathbb{P}^n$.
\item
If $l(R)\ge n$ and $X\not\simeq \mathbb{P}^n$, 
then $X\simeq\mathbb{P}(1,1,2,\ldots, 2)$. 
\end{enumerate}
\end{thm}

For the case where the associated extremal contraction is 
birational, we have the following estimates 
which are special cases of \cite[Theorem 3.2.1]{fujino-sato3}. 

\begin{thm}\label{lengthbirational}
Let $X$ be a $\mathbb{Q}$-factorial 
projective toric $n$-fold, and let $R$ be a $K_X$-negative 
extremal ray of $\NE(X)$. Suppose that 
the contraction morphism $\varphi_R:X\to W$ 
associated to $R$ is birational. 
Then, we obtain 
\[
l(R)<d+1, 
\]
where \[d=\max_{w\in W} \dim \varphi^{-1}_R(w)\leq n-1. \] 
When $d=n-1$, we have a sharper inequality 
\[
l(R)\leq d=n-1. 
\]
In particular, if 
$l(R)=n-1$ holds, 
then $\varphi_R:X\to W$ can be described as follows. 
There exists a torus invariant smooth point $P\in W$ such that 
$\varphi_R:X\to W$ 
is a weighted blow-up at $P$ with the weight $(1, a, \ldots, a)$ for some 
positive integer $a$. 
In this case, the exceptional locus $E$ of $\varphi_R$ is a torus invariant 
prime divisor and is isomorphic to $\mathbb P^{n-1}$. 
\end{thm}

This estimate shows that the extremal ray $R$ with $l(R)>n-1$ must be 
of fiber type. In this case, we can determine the structure of the 
associated contraction $\varphi_R$ as follows.

\begin{thm}\label{lengthfanotop1} 
Let $X$ be a $\mathbb Q$-factorial 
projective toric $n$-fold with $\rho (X)\geq 2$, and 
let $R$ be a $K_X$-negative extremal ray of 
$\NE(X)$. If 
$l(R)>n-1$, 
then the extremal contraction $\varphi_R:X\to W$ associated to 
$R$ is a $\mathbb P^{n-1}$-bundle over $\mathbb P^1$. 
\end{thm}

\begin{rem}
Theorem \ref{lengthfanotop1} holds for 
projective $\mathbb{Q}$-Gorenstein toric varieties 
(without the assumption that $X$ is $\mathbb{Q}$-factorial). 
For the details, please see \cite[Proposition 3.2.9]{fujino-sato3}. 
\end{rem}

As a generalization of Theorem \ref{lengthfanotop1}, we prove 
the following theorem about the structure of 
extremal contractions of fiber type. 
More precisely, we will prove a sharper result in Section \ref{f-sec3} 
(see Theorem \ref{maintheorem}). Theorem \ref{intromain} is a direct 
easy consequence of Theorem \ref{maintheorem} (see Corollary \ref{maincor}). 

\begin{thm}[Main theorem]\label{intromain}
Let $X$ be a $\mathbb{Q}$-factorial projective toric $n$-fold. 
Let $\varphi_R:X\to W$ be a Fano contraction associated to a $K_X$-negative 
extremal ray $R\subset\NE(X)$ such that 
the dimension of a fiber of $\varphi_R$ is $d$, equivalently, 
$d=\dim X-\dim W$. 
If $l(R)>d$, then $\varphi_R$ is 
a $\mathbb{P}^d$-bundle over $W$. 
\end{thm}

We show that this result is sharp by Examples 
\ref{primbutnotbdl} and 
\ref{wtproj}. 
We note that Theorem \ref{intromain} is 
nothing but Theorem \ref{rho1} (1) if $\dim W=0$. 
Therefore, we can see Theorem \ref{intromain} as a 
generalization of Theorem \ref{rho1} (1). 

\begin{ack}
The authors would like to 
thank the referee for useful comments. 
\end{ack}

\section{Preliminaries}
In this section, we introduce some basic results and notation of 
the toric geometry in order to prove the main theorem. 
For the details, please see \cite{cls}, 
\cite{fulton} and  \cite{oda}.
See also \cite{fujino-sato}, \cite[Chapter 14]{matsuki} 
and \cite{reid} for the toric Mori theory.

\medskip

Let $X=X_\Sigma$ be the toric $n$-fold associated to a fan 
$\Sigma$ in $N=\mathbb{Z}^n$ 
over an algebraically closed field $k$ of arbitrary characteristic. 
We will use the notation $\Sigma=\Sigma_X$ to denote the fan 
associated to a 
toric variety $X$. 
It is well known that there exists a one-to-one correspondence between 
the $r$-dimensional cones in $\Sigma$ and the torus invariant 
subvarieties of dimension $n-r$ in $X$. Let $\G(\Sigma)$ 
be the set of primitive generators for $1$-dimensional cones in $\Sigma$. 
Thus, for $v\in\G(\Sigma)$, 
we have a torus invariant prime divisor corresponding to $v$. 

\medskip

For an $r$-dimensional simplicial cone $\sigma\in\Sigma$, 
let $N_\sigma\subset N$ be the sublattice generated by $\sigma\cap N$ 
and let $\sigma\cap\G(\Sigma)=\{v_1,\ldots,v_r\}$, that is, 
$\sigma=\langle v_1,\ldots,v_r\rangle$, where 
$\langle v_1,\ldots,v_r\rangle$ is the $r$-dimensional 
strongly convex cone generated by $\{v_1,\ldots,v_r\}$. 
Put 
\[
\mult(\sigma):=[N_\sigma:\mathbb{Z}v_1+\cdots+\mathbb{Z}v_r]
\]
which is the index of the subgroup $\mathbb{Z}v_1+\cdots+\mathbb{Z}v_r$ 
in $N_\sigma$. 
The following property is fundamental. 

\begin{prop}\label{kotensu}
Let $X$ be a $\mathbb{Q}$-factorial toric $n$-fold, and let $\tau\in\Sigma$ 
be an $(n-1)$-dimensional cone and $v\in\G(\Sigma)$. If 
$v$ and $\tau$ generate a maximal cone $\sigma$ in $\Sigma$, then 
\[
D\cdot C=\frac{\mult(\tau)}{\mult(\sigma)},
\]
where $D$ is the torus invariant prime divisor corresponding to $v$, while 
$C$ is the torus invariant curve corresponding to $\tau$.
\end{prop}

Let $X$ be a projective toric variety. 
We put 
\[
\mathrm{Z}_1(X):=\{1\text{-cycles of} \ X\}, 
\]
and 
\[
\mathrm{Z}_1(X)_{\mathbb R}:=
\mathrm{Z}_1(X)\otimes \mathbb R.  
\]
Let 
\[
\Pic (X)\times \mathrm{Z}_1(X)
\to \mathbb Z
\] 
be a pairing 
defined by 
$(\mathcal L, C)\mapsto \deg _C\mathcal L$. 
By extending it 
by bilinearity, we have a pairing 
\[
(\Pic (X)\otimes \mathbb R)\times \mathrm{Z}_1(X)_{\mathbb R}
\to \mathbb R. 
\] 
We define 
\[
\mathrm{N}^1(X):=(\Pic (X)\otimes \mathbb R)/\equiv
\]
and 
\[
\mathrm{N}_1(X):= \mathrm{Z}_1(X)_{\mathbb R}/\equiv, 
\] 
where the {\em numerical equivalence} $\equiv$ 
is by definition the smallest equivalence relation which 
makes $\mathrm{N}^1$ and $\mathrm{N}_1$ into dual 
spaces. 

Inside $\mathrm{N}_1(X)$ there is a distinguished cone of effective 
$1$-cycles of $X$, 
\[
{\NE}(X)=\left\{\, Z\, \left| \ Z\equiv 
\sum a_iC_i \ \text{with}\ a_i\in \mathbb R_{\geq 0}\right.\right\}
\subset \mathrm{N}_1(X), 
\] 
which is usually called the {\em{Kleiman--Mori cone}} of $X$. 
It is known that $\NE(X)$ is a rational polyhedral cone. 
A face $F\subset {\NE}(X)$ is called an {\em{extremal face}} 
in this case. 
A one-dimensional extremal face is called an {\em{extremal 
ray}}. 


\medskip

Next, we introduce a combinatorial description of toric Fano 
contractions which are main objects of this paper. 
Let $X=X_\Sigma$ be a $\mathbb{Q}$-factorial projective 
toric $n$-fold and 
$\varphi_R:X\to W$ be the extremal contraction 
associated to an extremal ray $R\subset\NE(X)$ of fiber type. 
Put 
\[
d:=\dim X-\dim W. 
\] 
Up to automorphisms of $N$, 
$\Sigma$ is constructed as follows:

For the standard basis $\{e_1,\ldots,e_{n}\}\subset N=\mathbb{Z}^n$, 
put $N':=\mathbb{Z}e_1+\cdots+\mathbb{Z}e_d$, while 
$N'':=\mathbb{Z}e_{d+1}+\cdots+\mathbb{Z}e_{n}$, that is, 
$N=N'\oplus N''$. 
Then, there exist $\{v_1,\ldots,v_{d+1}\}\subset\G(\Sigma)\cap N'$ 
such that  $\{v_1,\ldots,v_{d+1}\}\setminus\{v_i\}$ 
generates a $d$-dimensional cone $\sigma_i\in\Sigma$ for any $1\le i\le d+1$, 
and $\sigma_1\cup\cdots\cup\sigma_{d+1}=N'\otimes\mathbb{R}$. 
Namely, we obtain the complete fan $\Sigma_F$ 
in $N'$ whose maximal cones are 
$\sigma_1,\ldots,\sigma_{d+1}$. $\Sigma_F$ is associated to a 
general fiber $F$ of $\varphi_R$, and the Picard number $\rho(F)$ is $1$. 
Moreover, for any $\{y_1,\ldots,y_{n-d}\}\subset\G(\Sigma)
\setminus\{v_1,\ldots,v_{d+1}\}$ which generates 
an $(n-d)$-dimensional cone in $\Sigma$, 
$\{v_1,\ldots,v_{d+1},y_1,\ldots,y_{n-d}\}\setminus\{v_i\}$ 
generates a maximal cone in $\Sigma$ for any $1\le i\le d+1$. 
Thus, the projection $N=N'\oplus N''\to N''$ induces $\varphi_R$. 

\begin{rem}\label{fanofiber}
This description shows that for a toric Fano contraction $\varphi_R:X\to W$, 
the dimension of any fiber is constant. 
As we saw above, the general fiber $F$ of $\varphi_R$ is a 
projective $\mathbb Q$-factorial toric variety of Picard number 
$\rho(F)=1$. 
Moreover, it is known that the fiber 
$\varphi_R^{-1}(w)_{\mathrm{red}}$ 
with the reduced structure is 
isomorphic to $F$ for every closed point $w\in W$ 
(see \cite[Proposition 15.4.5]{cls} and 
\cite[Corollary 14-2-2]{matsuki}).
\end{rem}
 
\section{Fano contractions}\label{f-sec3}

The following result is the main theorem of this paper.  

\begin{thm}\label{maintheorem}
Let $X=X_\Sigma$ be a $\mathbb{Q}$-factorial projective toric $n$-fold. 
Let $\varphi_R:X\to W$ be a Fano contraction associated to a $K_X$-negative 
extremal ray $R\subset\NE(X)$, and $d=n-\dim W$ 
be the dimension of a fiber of 
$\varphi_R$. 
If a general fiber of $\varphi_R$ is isomorphic to $\mathbb{P}^d$ and 
\[
-K_X\cdot C>\frac{d+1}{2}
\]
holds 
for any curve $C$ on $X$ contracted by $\varphi_R$, 
then $\varphi_R$ is 
a $\mathbb{P}^d$-bundle over $W$. 
\end{thm}

\begin{proof}
We may assume that $\varphi_R:X\to W$ is induced by the following projection:
\[
\begin{array}{ccc}
N=\mathbb{Z}^n & \stackrel{p}{\longrightarrow} & \mathbb{Z}^{n-d} \\
\rotatebox{90}{$\in$} & & \rotatebox{90}{$\in$} \\
(x_1,\ldots,x_n) & \longmapsto & (x_{d+1},\ldots,x_n).
\end{array}
\]
Let $\{e_1,\ldots,e_n\}$ be the standard basis for $N=\mathbb{Z}^n$. 
We put $$v_1:=e_1,\quad \ldots,\quad v_d:=e_d,\quad \text{and}\quad 
v_{d+1}:=-(e_1+\cdots+e_d). $$
Then $\Sigma$ contains the $d$-dimensional 
subfan $\Sigma_F$ corresponding to a general fiber $F\simeq \mathbb{P}^d$ 
whose maximal cones are 
\[
\left\langle\left\{v_1,\ldots,v_{d+1}\right\}
\setminus\{v_i\}\right\rangle\quad (1\le i\le d+1).
\]

Let $V_\sigma\subset N\otimes_\mathbb{Z}\mathbb{R}$ 
be the linear subspace spanned by $\sigma$ 
for any $(n-d)$-dimensional 
cone $\sigma$ in $\Sigma$ such that $\left(\sigma\cap\G(\Sigma)\right)\cap
\{v_1,\ldots,v_{d+1}\}=\emptyset$.  
Then it is sufficient to show that 
\begin{equation}\label{eq1}
V_\sigma\cap\mathbb{Z}^{n}\stackrel{p}{\longrightarrow}\mathbb{Z}^{n-d}
\end{equation}
is bijective. 
This is because the restriction of $\varphi_R:X\to W$ to the 
affine toric open subset $U$ corresponding to an $(n-d)$-dimensional 
cone $p(\sigma)$ is the 
second projection $\mathbb P^d\times U\to U$ if $p$ in 
\eqref{eq1} is bijective. 
The injectivity of \eqref{eq1} 
is trivial. Therefore, we will show the surjectivity of \eqref{eq1}. 

Let $y_1,\ldots,y_{n-d}\in\G(\Sigma)\setminus\{v_1,\ldots,v_{d+1}\}$ be 
the primitive generators for any $(n-d)$-dimensional cone in $\Sigma$ 
such that $p(\langle y_1,\ldots,y_{n-d}\rangle)$ is also $(n-d)$-dimensional. 
Put 
\[
\begin{split}
y_1&=(b_{1,1},\ldots,b_{d,1},a_{1,1},\ldots,a_{n-d,1}),
\\ 
&\vdots\\
y_{n-d}&=(b_{1,n-d},\ldots,b_{d,n-d},a_{1,n-d},\ldots,a_{n-d,n-d}).
\end{split}
\]
For any $(z_1,\ldots,z_{n-d})\in\mathbb{Z}^{n-d}$, 
we can take $(c_1,\ldots,c_{n-d})\in\mathbb{R}^{n-d}$ 
satisfying  
\[
p(c_1y_1+\cdots+c_{n-d}y_{n-d})=
c_1p(y_1)+\cdots+c_{n-d}p(y_{n-d})=
(z_1,\ldots,z_{n-d}).
\]
We note that the matrix
\[
  A: = \left(
    \begin{array}{ccc}
      a_{1,1}  & \ldots & a_{1,n-d} \\
      \vdots  & \ddots & \vdots \\
      a_{n-d,1}  & \ldots & a_{n-d,n-d}
    \end{array}
  \right)
\]
is regular as a real matrix because $p(y_1),\ldots,p(y_{n-d})$ 
generates an $(n-d)$-dimensional 
cone. 
Therefore, $(c_1,\ldots,c_{n-d})$ is uniquely determined 
by 
\[
  \left(
    \begin{array}{c}
      c_1  \\
      \vdots   \\
      c_{n-d}  
    \end{array}
  \right)
  =A^{-1}
    \left(
    \begin{array}{c}
      z_1  \\
      \vdots   \\
      z_{n-d}  
    \end{array}
  \right)
  \in\mathbb{Q}^{n-d}.
\]
Thus, all we have to do is to show that 
\[
c_1b_{r,1}+\cdots+c_{n-d}b_{r,n-d}\in\mathbb{Z}
\]
for any $1\le r\le d$. 

By considering the principal Cartier divisors of the dual basis of 
$\{e_1,\ldots,e_n\}$, 
we obtain the relations 
\begin{eqnarray}\label{rationalfunc}
  \left\{
    \begin{array}{rcl}
 D_1-D_{d+1}+b_{1,1}E_1+\cdots+b_{1,n-d}E_{n-d}+H_1 & = & 0, \\
     & \vdots & \\
 D_d-D_{d+1}+b_{d,1}E_1+\cdots+b_{d,n-d}E_{n-d}+H_d & = & 0, \\
 a_{1,1}E_1+\cdots+a_{1,n-d}E_{n-d}+H_{d+1} & = & 0, \\
 & \vdots & \\
 a_{n-d,1}E_1+\cdots+a_{n-d,n-d}E_{n-d}+H_{n} & =& 0
    \end{array}
  \right.
\end{eqnarray}
in $\mathrm{N}^1(X)$, 
where $D_1,\ldots,D_{d+1},E_1,\ldots,E_{n-d}$ are the torus invariant 
prime divisors corresponding to 
$v_1,\ldots,v_{d+1},y_1,\ldots,y_{n-d}$, respectively, and $H_1,\ldots,H_n$ are 
some linear combinations of torus invariant prime divisors other 
than $D_1,\ldots,D_{d+1},E_1,\ldots,E_{n-d}$. 
Let $C=C_{r}$ $(1\le r\le d)$ be the torus invariant curve corresponding to 
the $(n-1)$-dimensional cone 
\[
\left\langle \left\{v_1,\ldots,v_{d},y_1,\ldots,y_{n-d}\right\}
\setminus\{v_r\}\right\rangle.
\]
Since $H_i\cdot C=0$ for any $1\le i\le n$, we may ignore $H_1,\ldots,H_n$ 
in the following calculation. 
Since the matrix $A$ is regular, 
we have 
\[
E_1\cdot C=\cdots=E_{n-d}\cdot C=0, 
\]
and 
\[
D_1\cdot C=D_2\cdot C=\cdots=D_{d+1}\cdot C
\]
by the above equalities \eqref{rationalfunc} in $N^1(X)$. 
Thus, we obtain 
\[
-K_X\cdot C=(d+1)D_i\cdot C
\]
for any $1\le i\le d+1$. 

Put 
\[
\alpha:=\mult\left(\left\langle\left\{v_1,\dots,v_{d},y_1,\ldots,y_{n-d}\right\}
\setminus\{v_r\}\right\rangle\right)
\]
and 
\[
\beta:=\mult\left(\left\langle\left\{v_1,\dots,v_{d},y_1,\ldots,y_{n-d}\right\}\right\rangle\right). 
\]
Then we get  
\[
D_r\cdot C=\frac{\alpha}{\beta}
\]
by Proposition \ref{kotensu}. We note that 
$\alpha\mid\beta$ always holds. 
Obviously, $\beta=|\det A|$. 
On the other hand, $\alpha$ is the product of 
the elementary divisors of the $n\times (n-1)$ matrix
\[
\left(
{}^{\rm t}v_1,
\ldots,
\stackrel{\vee}{{}^{\rm t}v_r},
\ldots,
{}^{\rm t}v_d,
{}^{\rm t}y_1,
\ldots,
{}^{\rm t}y_{n-d}
\right)
=
\left(
    \begin{array}{ccccccccc}
 1 & & &  & & &  b_{1,1} & \ldots & b_{1,n-d} \\
 &  \ddots & &  & \text{{\huge{0}}} &  &   \vdots  & \ddots & \vdots \\
 & & 1 & & & &   b_{r-1,1} & \ldots & b_{r-1,n-d} \\
  0 & \cdots & 0 & 0 & \cdots & 0 &  b_{r,1} & \ldots & b_{r,n-d} \\
  & & & 1 & & & b_{r+1,1} & \ldots & b_{r+1,n-d} \\
    & & & & \ddots & &   \vdots  & \ddots & \vdots \\
  & & & & & 1 &   b_{d,1} & \ldots & b_{d,n-d} \\
  & & & & & &   a_{1,1}  & \ldots & a_{1,n-d} \\
  & \text{{\huge{0}}} &  & & & &    \vdots  & \ddots & \vdots \\
  & & & & & &    a_{n-d,1}  & \ldots & a_{n-d,n-d}
    \end{array}
  \right),
\]
where ${}^{\rm t}v$ stands 
for the transpose of $v$. 
By interchanging rows of this matrix, 
one can easily check that $\alpha$ is also the product of 
the elementary divisors of the $(n-d+1)\times (n-d)$ matrix
\[
  \overline{A} = \left(
    \begin{array}{ccc}
    b_{r,1} & \ldots & b_{r,n-d} \\
      a_{1,1}  & \ldots & a_{1,n-d} \\
      \vdots  & \ddots & \vdots \\
      a_{n-d,1}  & \ldots & a_{n-d,n-d}
    \end{array}
  \right).
\]


Suppose that $D_r\cdot C<1$ holds. 
Then, more strongly, we obtain the inequality $D_r\cdot C\le \frac{1}{2}$ 
by the relation $\alpha\mid\beta$. 
Thus, the following inequality 
\[
 -K_X\cdot C=(d+1)D_r\cdot C\le \frac{d+1}{2}
\] 
holds.  
However, this contradicts the assumption that 
$\frac{d+1}{2}< -K_X\cdot C$. Therefore, the equality 
$$\frac{\alpha}{\beta}=D_r\cdot C=1$$ must always hold. 
Since the general theory of elementary divisors says that 
$\alpha$ is the greatest common divisor of the 
$(n-d)\times(n-d)$ minor determinants of $\overline{A}$, 
the $(n-d)\times(n-d)$ determinant
\[
  \left|
    \begin{array}{ccc}
    b_{r,1}  & \ldots & b_{r,n-d} \\
      a_{1,1}  & \ldots & a_{1,n-d} \\
      \vdots  & \vdots & \vdots \\
      a_{i-1,1}  & \ldots & a_{i-1,n-d} \\
      a_{i+1,1}  & \ldots & a_{i+1,n-d} \\
          \vdots  & \vdots & \vdots \\
      a_{n-d,1}  & \ldots & a_{n-d,n-d}
    \end{array}
  \right|
\]
is divisible by $\det A$ for any $1\le i\le n-d$. 
Let 
\[
 \widetilde{A}: = \left(
    \begin{array}{ccc}
      \widetilde{a}_{1,1}  & \ldots &  \widetilde{a}_{1,n-d} \\
      \vdots  & \ddots & \vdots \\
       \widetilde{a}_{n-d,1}  & \ldots &  \widetilde{a}_{n-d,n-d}
    \end{array}
  \right)
\]
be the cofactor matrix of $A$. Then, 
\[
c_1b_{r,1}+\cdots+c_{n-d}b_{r,n-d}
\]
\[
=
\frac{1}{\det A}
\left(\widetilde{a}_{1,1}z_1+\cdots+\widetilde{a}_{1,n-d}z_{n-d}\right)b_{r,1}
+\cdots+
\frac{1}{\det A}
\left(\widetilde{a}_{n-d,1}z_1+\cdots+\widetilde{a}_{n-d,n-d}z_{n-d}\right)b_{r,n-d}
\]
\[
=
\frac{
\widetilde{a}_{1,1}b_{r,1}+\cdots+\widetilde{a}_{n-d,1}b_{r,n-d}}
{\det A}\times
z_1
+\cdots+
\frac{
\widetilde{a}_{1, n-d}b_{r,1}+\cdots+\widetilde{a}_{n-d,n-d}b_{r,n-d}}
{\det A}\times
z_{n-d}
\]
is an integer. This completes the proof.
\end{proof}

The following example shows that Theorem \ref{maintheorem} is sharp. 

\begin{ex}\label{primbutnotbdl}
Let $\{e_1,\ldots,e_n\}$ be the standard basis for $N=\mathbb{Z}^n$ 
and $p:N\to\mathbb{Z}^{n-d}$ be the projection
\[
(x_1,\ldots,x_d,x_{d+1},\ldots,x_n)\mapsto (x_{d+1},\ldots,x_n)
\]
for $1\le d< n$. 
Put 
\[
v_1:=e_1,\ \ldots,\ v_d:=e_d,\ v_{d+1}:=-(e_1+\cdots+e_d),
\] 
\[
y_1:=e_{d+1},\ \ldots,\ y_{n-d-1}:=e_{n-1},\ 
y_{n-d}:=e_1+e_{d+1}+\cdots+e_{n-1}+2e_n. 
\]
Let $\Sigma$ be the fan in $N$ whose maximal cones are 
generated by 
$\{v_1,\ldots,v_{d+1},y_1,\ldots,y_{n-d}\}\setminus\{v_i\}$ 
for $1\le i\le d+1$. 
In this case, $X=X_\Sigma$ has a Fano contraction 
whose general fiber is isomorphic to $\mathbb{P}^d$. 
Moreover, every fiber with the reduced structure 
is isomorphic to $\mathbb P^d$ (see Remark \ref{fanofiber}). 
However, 
$X$ does not decompose into 
$\mathbb{P}^d$ and a toric affine $(n-d)$-fold, because 
\[
\frac{p(y_1)+\cdots+p(y_{n-d})}{2}=
e_{d+1}+\cdots+e_n\in\mathbb{Z}^{n-d},
\]
while
\[
\frac{y_1+\cdots+y_{n-d}}{2}=
\frac{1}{2}e_1+e_{d+1}+\cdots+e_n\not\in N.
\]
From this noncomplete variety, 
one can easily construct a projective toric $n$-fold 
which has a Fano contraction associated to an 
extremal ray of length $\frac{d+1}{2}$ 
(for example, add the generator $y_{n-d+1}:=
-(e_{d+1}+\cdots+e_n)$ and 
compactify $\Sigma$). 
\end{ex}

If we make the inequality in Theorem \ref{maintheorem} 
stronger, then the assumption that a general fiber of a 
Fano contraction is isomorphic 
to the projective space automatically holds as follows. 

\begin{cor}\label{maincor}
Let $X=X_\Sigma$ be a $\mathbb{Q}$-factorial projective toric $n$-fold. 
Let $\varphi_R:X\to W$ be a Fano contraction associated to a $K_X$-negative 
extremal ray $R\subset\NE(X)$, and $d=n-\dim W$ 
be the dimension of a fiber of 
$\varphi_R$. 
If $-K_X\cdot C>d$ holds 
for any curve $C$ on $X$ contracted by $\varphi_R$, 
then $\varphi_R$ is 
a $\mathbb{P}^d$-bundle over $W$. 
\end{cor}

\begin{proof}
Let $F$ be a general fiber of $\varphi_R$ and let $C$ be 
any curve on $F$. 
Then, by adjunction, we have 
\[
d< -K_X\cdot C= -K_F\cdot C. 
\]
Therefore, by Theorem \ref{rho1} (1), $F\simeq \mathbb{P}^d$ holds. 
Since $\frac{d+1}{2}\le d$, we can apply Theorem \ref{maintheorem}. 
\end{proof}

As an easy consequence of Corollary \ref{maincor}, we obtain: 

\begin{cor}
Let $X=X_\Sigma$ be a $\mathbb Q$-factorial 
projective toric $n$-fold and let $\Delta$ be 
any effective {\em{(}}not necessarily torus invariant{\em{)}} 
$\mathbb R$-divisor on $X$. 
Let $\varphi_R: X\to W$ be a Fano 
contraction associated to a $(K_X+\Delta)$-negative 
extremal ray $R\subset \NE (X)$ 
with $d=n-\dim W$.  
If $-(K_X+\Delta)\cdot C>d$ for any curve 
$C$ on $X$ contracted by $\varphi_R$, 
then $\varphi_R$ is a $\mathbb P^d$-bundle over 
$W$.   
\end{cor}

\begin{proof}
We can easily see that $D\cdot C\geq 0$ for any effective Weil divisor 
$D$ on $X$ and any curve $C$ on $X$ contracted by $\varphi_R$ 
since $\varphi_R: X\to W$ is a toric Fano contraction 
of a $\mathbb Q$-factorial projective toric variety $X$. 
Therefore, we get 
$$
d< -(K_X+\Delta)\cdot C\leq -K_X\cdot C
$$ 
for any curve $C$ on $X$ contracted by $\varphi_R$. 
Thus, we see that $\varphi_R:X\to W$ is a $\mathbb P^d$-bundle 
over $W$ by Corollary \ref{maincor}. 
\end{proof}

The following example shows that 
Corollary \ref{maincor} is sharp.

\begin{ex}\label{wtproj}
Let $F:=\mathbb{P}(1,1,2,\ldots,2)$ be the $d$-dimensional 
weighted projective space and 
$W$ a $\mathbb{Q}$-factorial projective toric $(n-d)$-fold. 
Then, the length of 
the extremal ray corresponding to the first projection 
$\varphi:X=W\times F\to W$ is $d$ 
(see \cite[Proposition 2.1]{fujino-osaka} and 
\cite[Proposition 3.1.6]{fujino-sato3}). 
\end{ex}


\end{document}